\documentclass[11pt]{amsart}
\usepackage{amsmath,amssymb,amsthm, mathtools}
\usepackage{amsaddr}
\usepackage{graphics}
\usepackage{graphicx}
\usepackage{epstopdf}
\usepackage{hyperref}
\usepackage{epsfig,color}

\newtheorem{conjecture}{Conjecture}
\newtheorem{theorem}[conjecture]{Theorem}
\newtheorem{lemma}[conjecture]{Lemma}
\newtheorem{proposition}[conjecture]{Proposition}

\newcommand \Aip {{{\mathcal A}_{i+}}}
\newcommand \Aim {{{\mathcal A}_{i-}}}
\newcommand \B {{\mathcal B}}
\newcommand \Bip {{{\mathcal B}_{i+}}}
\newcommand \Bim {{{\mathcal B}_{i-}}}

\newcommand \Po {{\mathcal P}}
\newcommand{\dw}{\downarrow\!}
\newcommand{\up}{\uparrow\!\!}
\newcommand{\Pn}{\mathcal{P}(n)}
\newcommand{\Bn}{\mathcal{B}(n)}
\newcommand{\Bi}{\mathcal{B}_i}
\newcommand{\I}{\mathcal{I}}
\newcommand{\Le}{\mathcal{L}}
\newcommand{\C}{\mathcal{C}}
\newcommand{\co}{\mathrm{comp}}
\newcommand{\D}{\mathcal{D}}
\newcommand{\HH}{\mathcal{H}}

\newcommand{\Ch}{\mathbf{C}}
\newcommand{\F}{\mathcal{F}}

\newcommand{\Gi}{\mathcal{G}_i}
\newcommand{\Gk}{\mathcal{G}_k}
\newcommand{\Nn}{\mathcal{N}}

\newcommand{\A}{\mathcal{A}}
\newcommand{\Bj}{\mathcal{B}_j}

\newcommand{\2}{\mathbf{2}^n}

\begin{document}

\parindent = 0cm
\parskip = .3cm

\title[The width of downsets ]{The width of downsets\vspace{-.5cm}}

\author{Dwight Duffus\vspace{-.5cm}}
\address {Mathematics Department\\
  Emory University, Atlanta, GA  30322 USA\vspace{-.5cm}}
   \email{dwightduffus@emory.edu}
\author{David Howard\vspace{-.5cm}}
\address{Mathematics Department\\
  Colgate University, Hamilton, NY 13346 USA\vspace{-.5cm}}
  \email{dmhoward@colgate.edu}
\author{Imre Leader\vspace{-.5cm}}
\address{Centre for Mathematical Sciences\\
  University of Cambridge, Cambridge CB3 0WB UK}
  \email{leader@dpmms.cam.ac.uk}
\date{\today}
\let\thefootnote\relax\footnote{Accepted by the European Journal of Combinatorics - DOI 10.1016/j.ejc.2018.11.005}
\footnote{\copyright 2019. This manuscript version is made available under the CC BY-NC-ND 4.0 license}

\begin{abstract}
How large an antichain can we find inside a given downset in the Boolean lattice $\Bn$? Sperner's 
theorem asserts that the largest antichain in the whole of $\Bn$ has size $\binom{n}{\lfloor n/2\rfloor}$; 
what happens for general downsets?\\

\noindent
Our main results are a Dilworth-type decomposition theorem for downsets, and a new proof of a result of
Engel and Leck that determines the largest possible 
antichain size over all downsets of a given size. We also prove some related results, such as determining 
the maximum size of an antichain inside the downset that we conjecture minimizes this quantity among 
downsets of a given size. 
\end{abstract}

\maketitle

\thispagestyle{empty}

\section{Introduction}\label{S:intro}

The {\it width} $w(X)$ of a finite partially ordered set $X$ is the maximum size of an antichain in $X$.  
By Sperner's theorem \cite{S}, the width of the Boolean lattice $\Bn$, the set of subsets of $[n] = \{1, 2, \ldots, n\}$, 
ordered by containment, is the maximum size of a level, namely $\binom{n}{\lfloor n/2\rfloor}$.  In this paper, we are 
interested in the relationship between the width and size of a {\it downset} $\D$ of $\Bn$, meaning a family 
of sets such that if $X \in \D$ and $Y \subseteq X$ then $Y \in \D$.  

One of our aims is to determine, for $n$ fixed, the maximum width of a downset of a given size 
(see Section \ref{S:peculiar}).  This turns out to be given by the initial segment of that size of a slightly 
nonstandard (total) order on $\Bn$ (see Theorem \ref{T:maxwidth}).  This was proved by Engel and
Leck \cite{EL}. Our proof is rather different, being based on compressions.

On the other hand, to {\it minimize} the width of a downset in $\Bn$ of given size $d$, we have a conjecture that 
the initial segment of $\Bn$ of size $d$ under the {\it binary order} realizes the minimum width (see 
Conjecture \ref{c:binary-min-width}).  This conjecture has been made independently by Goldwasser \cite{G}.  
We work out the width of this downset (with an argument that is 
perhaps more involved that it ought to be -- see Proposition \ref{P:Goldwasser}).  Here we are motivated by 
a beautiful 35-year-old conjecture of Daykin and Frankl \cite{DF} for convex subsets $\C$ of $\Bn$.  Recall 
that $\C$ is {\it convex} in $\Bn$ if whenever $X, Y \in \C$ with $X \subseteq Z \subseteq Y$ then also 
$Z \in \C$.

\begin{conjecture} {\rm(Daykin and Frankl \label{c:convex}\cite{DF})}
For any nonempty convex subset $\C$ of  $\Bn$,\\
$$\frac{w(\C)}{|\C|}\ \ge \ \frac{\binom{n}{\lfloor n/2 \rfloor}}{2^n} \ .$$
\end{conjecture}

We verify this when  $\C$ is an initial segment of the binary order (see Theorem \ref{T:alpha}).  Thus, 
Conjecture \ref{c:convex} specialized to downsets would follow from Conjecture \ref{c:binary-min-width}.

Still concerning convex subsets $\C$ of $\Bn$, Dilworth's theorem \cite{D} says that any convex set (indeed, any 
partially ordered set) has a partition into $w(\C)$-many chains.  In the case of convex subsets, we could 
ask for much more, namely, that the chains are {\it skipless}, meaning that they skip no levels of $\Bn$ 
(in other words, successive elements
of a chain increase in size by exactly 1).  We show that this in indeed the case (Theorem \ref{T:SD}).   As an
application, we determine precisely when adding an element to a downset of $\Bn$ increases its width 
(Proposition \ref{P:general}).

We also consider a related problem. If we have a given number of $r$-sets, and we wish to minimize the 
size of the downset they generate, then by the Kruskal-Katona theorem (\cite{Ka}, \cite {Kr}, cf. \cite{BB}) 
we should take an initial segment of $[n]^{(r)}$ (the family of all $r$-sets from $[n]$) under the colexicographic 
order.  Thus,  the size of the 
downset is independent of $n$. Now, if we instead wished to {\it maximize} the size of the downset, then 
this is not a sensible question, as we would just take some disjoint $r$-sets. However, this is in some sense 
cheating, because the downset generated has much larger antichains that the original family of $r$-sets.

So a more natural question is as follows. Call a family of $r$-sets {\it top-heavy} or simply {\it heavy} if there 
is no larger antichain in the downset it generates. And now the question would be: among heavy families of 
$r$-sets of given size, which one generates the largest downset? Here we are allowing $n$ to vary.  We make a 
conjecture on this value, and give some (rather weak) bounds. 

The paper is organized as follows.  In Section \ref{S:peculiar}, we find the maximum width of a downset
of given size in $\Bn$. Section \ref{S:binary} 
contains our results and conjectures on the minimization problem and its relation to the Daykin-Frankl 
conjecture.  In Section \ref{S:SD} we prove that every convex subset of $\Bn$ has a partition into 
width-many skipless chains.  This result is then applied in Section \ref{S:downsets} to describe when 
the addition of a single new element increases the width of a downset in $\Bn$.   Finally, in 
Section \ref{S:shadows} we consider the problem about heavy families described above.

Combinatorial terms and notation are standard -- see e.g. Bollob\'as \cite{BB} for these and further background.

\section{The Maximum Width of a Downset}\label{S:peculiar}

Among all downsets of $\Bn$ of given cardinality, which one {\it maximizes} the width?
The answer is that we should take initial segments of some ordering, but interestingly it
is not one of the `standard' orderings on $\Bn$. 

Recall that in the {\it binary} ordering on $\Bn$ we have $A<B$ if $\max (A \triangle B) \in B$, 
and that in the {\it simplicial} ordering  on $\Bn$ we have $A<B$ if either $|A| < |B|$ or else 
$|A| = |B|$ and $\min (A \triangle B) \in A$. Thus in the binary ordering we `go up in
subcubes', and also the restriction to a level is the colex order, while in the
simplicial ordering we `go up in levels', with the restriction to a level being the
lex ordering. These are the standard two orderings on $\Bn$: for example, initial
segments of the simplicial ordering solve the vertex-isoperimetric problem while
initial segments of the binary ordering solve the edge-isoperimetric problem 
(see e.g. [1] for details).

In contrast, here we need a modification of the simplicial ordering. Let us
define the {\it level-colex} ordering on $\Bn$ by setting
$A<B$ if either $|A| < |B|$ or else $|A| = |B|$ and $\max (A \triangle B) \in B$. 
In other words, we go up in levels, but in each level we use colex instead of lex. 
Our aim is to give a direct proof of a lovely result of Engel and Leck \cite{EL} that,
among downsets of a given size, initial segments of the level-colex order maximize the width.  
The point is that, if we are going to take a downset 
with say all sets of size less than $k$ and also some $k$-sets, then we want those 
$k$-sets to have small shadow -- so by the Kruskal-Katona theorem (\cite{Ka}, \cite {Kr}, cf. \cite{BB}) 
we should take those $k$-sets to be an initial segment of the colex order.
Interestingly, we know of no other problem for which this level-colex ordering provides 
the extremal examples.

Part of the difficulty in proving this result arises from the fact that, in an initial segment of the level-colex 
ordering, the set of maximal elements may not form a maximum-sized antichain. The 
maximal elements do often achieve the width (for example, when we have all sets of size 
at most $k$ for $k \le (n + 1)/2$), but not always (for example, when our initial segment has size 1 greater 
than this). We also mention in passing that if one wished to prove the result only for 
certain sizes (namely when our initial segment consists of all sets of size at most $k$)
then other methods are available.

Our method is based on the use of `codimension-1 compressions', which were
originally introduced in \cite{BL}.
We need a small amount of notation.  It is slightly more convenient to view $\B(n)$
explicitly as a power set -- we will write $\Po(X)$ for the power set of a set $X$.
For a set system $\A$ on $[n]$ (i.e. $\A \subseteq \Bn$),
and $1 \le i \le n$, the {\it i-sections} of $\A$ are the set systems
on $[n] -\{i\}$ given by
$$\Aim=\{A \in \Po ([n]-\{i\}):\,A \in \A\},$$
$$\Aip=\{A \in \Po ([n]-\{i\}):\,A \cup \{i\} \in\A\}.$$ 

We can define the level-colex ordering on $\Po (X)$ whenever $X$ is
(totally) ordered -- again, $A$ precedes $B$ if either 
$|A| <  |B|$ or else $|A| = |B|$ and $\max (A \triangle B) \in B$. 
It is easy to see that if $\A$ is an initial segment of
the level-colex order on $\Bn$ then both $\Aim$ and $\Aip$ are initial segments of
the level-colex order on $\Po ([n]-\{i\})$.

For any $\A \subseteq \Bn$ and $1 \leq i \leq n$, we define a system $C_i(\A)
\subseteq \Bn$,
the {\it i-compression} of $\A$, by giving its $i$-sections:
$C_i(\A)_{i+}$ is the set of the first $|\A_{i+}|$ elements in the
level-colex order on $\Po([n]-\{i\})$, and similarly for
$C_i(\A)_{i-}$. 
In other words, $C_i$ `compresses'
each $i$-section of $\A$ into the level-colex order. We say that $\A$ is
{\it i-compressed} if $C_i(\A)=\A$. Thus for example
an initial segment of the level-colex order on
$\Bn$ is $i$-compressed for every $i$. Note that the $i$-compression of a
downset is again a downset (because any two initial segments of an ordering are
nested, in the sense that one is a subset of the other).

A natural question to ask is whether a set system that is $i$-compressed for all $i$
is necessarily an initial segment of the level-colex order. But in fact it is easy to see 
that this is not the case. For example, for $n=3$ we may take the set system
$\{\emptyset,\{1\}, \{2\}, \{1,2\}\}$.
However, and this is one the key properties of this kind of compression, it turns
out that this is essentially the unique such example.

\begin{lemma}\label{L:compression}
    Let $\A \subseteq \Bn$ be $i$-compressed for all $i$. Then
    either $\A$ is an initial segment of the level-colex order on $\Bn$, or else
    $n$ is odd (say $n=2r+1$) and     
    $$\A = [n]^{(\leq r)} - \{\{r+2, r+3, \ldots, n\}\} \cup \{\{1,2,\ldots,r+1\}\},$$    
     or else $n$ is even (say $n=2r$) and 
    $$\A=[n]^{(<r)} \cup \{A \in [n]^{(r)}:\; n \notin A\} - \{\{r,r+1, \ldots, n-1\}\} \cup \{\{1,2, \ldots, r-2,r-1,n\}\}.$$ 
\end{lemma}

\begin{proof}
Suppose that $\A$ is not an initial segment of the level-colex order on $\Bn$. Then 
there are sets $A,B \in \Bn$ with $A \in \A$, $B \notin \A$, and $B < A$ in the level-colex order. 
For any $i$, we cannot have $i \in A,B$ or $i \notin A,B$, since $\A$ is $i$-compressed. It follows that
$A=B^c$.

Thus, for any $A \in \A$, there is at most one $B < A$ such that $B \notin \A$,
namely $A^c$, and similarly, for any $B \notin \A$, there is at most one $A > B$
such that $A \in \A$. Taking $A$ to be the last set in $\A$, and $B$ to be the
first set not in $\A$, it follows immediately that $\A=\{C \in \Bn:\, C \leq A\} - \{B\}$, with 
$B$ the immediate predecessor of $A$ and $B=A^c$.  However, by the definition of the 
level-colex order, this can only happen in one case: if $n$ is odd, say $n=2r+1$, then 
$B$ must be the final $r$-set in colex, and if $n$ is even, say $n=2r$, then $B$ must 
be the final $r$-set in colex that does not contain $n$.
\end{proof}

We are now ready to prove that initial segments of the level-colex order maximize the width.

\begin{theorem}\label{T:maxwidth}
  Let $\A$ be a downset in $\Bn$, and let $\I$ be the set of the first $|\A|$ elements in the 
  level-colex order on $\Bn$.  Then $w(\A) \leq w(\I)$.
\end{theorem}

It turns out to be easier to deal with maximal elements instead of general antichains. So, for 
a downset $\A$, let us write $m(\A)$ for the number of maximal elements of $\A$. Because 
any antichain (in some downset) is the set of maximal elements of the downset it generates, 
Theorem \ref{T:maxwidth} will follow immediately from the following.

\begin{theorem}\label{T:maxmax}
   Let $\A$ be a downset in $\Bn$, and let $\I$ be the set of the first $|\A|$ elements in the 
   level-colex order on $\Bn$.  Then $m(\A) \leq m(\I)$.   
\end{theorem}
\begin{proof}
We proceed by induction on $n$. As the result is trivial for $n=1$, we turn to the induction step.
We first wish to show that for any $\A \in \Bn$, and any $1 \leq i \leq n$, we have $m(\A) \leq
m(C_i(\A))$, in other words that an $i$-compression does not decrease the number of maximal elements.

For convenience, write $\B$ for $C_i(\A)$. Now, the maximal elements of $\A$ consist of the 
maximal elements of $\Aip$ together with those maximal elements of $\Aim$ that do not belong to
$\Aip$. (Recall that $\Aip$ and $\Aim$ are subsets of $\B(n-1)$.)
And similarly for $\B$.

By the induction hypothesis, we have $m(\Bip) \geq m(\Aip)$ and $m(\Bim) \geq m(\Aim)$. 
Also, the maximal elements of an initial segment of the simplicial ordering form a final segment 
of that initial segment -- this is because the lower shadow of a colex initial segment is again a
colex initial segment. It follows that, if we consider the two initial segments $\Bip$ and $\Bim$, 
we must have that either every element of $\Bim - \Bip$ is a maximal element of $\Bim$, or 
every maximal element of $\Bim$ misses $\Bip$. In either case, we see that the set of maximal 
elements of $\Bim$ that do not belong to $\Bip$ is at least as large as the set of maximal elements 
of $\Aim$ that do not belong to $\Aip$. This establishes our claim.

Define a sequence of set systems $\A_0,\A_1,\ldots$ as follows. Set $\A_0=\A$. Having defined 
$\A_0,\ldots,\A_k$, if $\A_k$ is $i$-compressed for all $i$ then stop the sequence with $\A_k$. 
Otherwise, there is an $i$ for which $\A_k$ is not $i$-compressed. Set $\A_{k+1} = \C_i(\A_k)$, and continue 
inductively.

This sequence has to end in some $\A_l$, because, loosely speaking, if an operator $C_i$ moves 
a set then it moves it to a set which is earlier in the level-colex order on $\Bn$. 
The set system $\A'=\A_l$ satisfies $|\A'| = |\A|$ and $m(\A') \geq m(\A)$, and is $i$-compressed
for every $i$. It follows by Lemma \ref{L:compression}  that either $\A'$ is an initial segment
of the level-colex  order on $\Bn$, or else $n$ is odd (say $n=2r+1$) and 
$$\A = [n]^{(\leq r)} - \{\{r+2, r+3, \ldots, n\}\} \cup \{\{1, 2, \ldots, r+1\}\},$$
or  else $n$ is even (say $n=2r$) and                                                                                                                           
$$\A=[n]^{(<r)} \cup \{A \in [n]^{(r)}:\; n \notin A\} - \{\{r, r+1, \ldots, n-1\}\} \cup \{\{1,2,\ldots,r-2,r-1,n\}\}.$$
Thus, to complete the proof, it remains only to observe that in the
latter two cases we have $m(\A') \leq m(\I)$.
\end{proof}

\section{The Minimum Width of a Downset}\label{S:binary}

In this section, let $<_b$ denote the {\it binary order} on $\Bn$: as noted above,
for $A, B \subseteq [n]$, $A <_b B$ if $\max A \Delta B \in B$.   As above, we tend to refer to the
restriction of the binary order to a level as the colex order on that level. We refer to a downset of $\Bn$ 
that is an initial segment of the binary order as a {\it binary} downset.

Recall that $\Bn$ has a {\it symmetric chain decomposition} (SCD), that is, a partition into skipless chains
whose minimum $A$ and maximum $B$ satisfy $|A| + |B| = n$.   There are many constructions of SCDs
of $\Bn$ -- we use one due to Greene and Kleitman \cite{GK} that we now outline.  It is useful to regard
members of $\Bn$ as both subsets of $[n]$ and binary sequences indexed by $[n]$.  

Given a binary $n$-sequence, scan from left to right.  When a 0 is scanned, it is temporarily
unpaired.  When a 1 is scanned, it is paired to the rightmost unpaired 0 and both are now paired, or else 
there are none and the 1 is unpaired.  Given a set $A$, we move up its chain in the Greene-Kleitman
SCD by successively replacing unpaired 0's by 1's, from left to right.  We move down the chain 
by replacing unpaired 1's with 0's, right to left.  Here is an example of the procedure.  Begin
with the set $A = \{1, 2, 6, 8, 9\}$ in $\mathcal{B}(10)$.  The pairing procedure results in
$$1 1 0 \underbracket{0 \overbracket{0 1} \overbracket{0 1} 1} 0$$
with unpaired 1's in positions 1 and 2, and unpaired 0's in positions 3 and 10.  Here are the predecessors and
the successors of $A$ in its chain with altered entries underlined:
\begin{align*}
A:\  & 1 1 0 0 0 1 0 1 1 0                    &   1 1 1 0 0 1 0 1 1 \underline{1}\\
& 1 \underline{0} 0 0 0 1 0 1 1 0      &    1 1 \underline{1} 0 0 1 0 1 1 0\\
& \underline{0} 0 0 0 0 1 0 1 1 0      &     A:\  1 1 0 0 0 1 0 1 1 0
\end{align*}
We refer to the minimum elements of the symmetric chains in the Greene-Kleitman SCD $\Ch$  as {\it special points}.  
These are exactly the subsets of $[n]$ with no unpaired 1's.  Thus, from above, $B = \{6, 8, 9\}$ is a special 
point in $\mathcal{B}(10)$.

Given a binary downset $\D$ of $\Bn$, let $s(\D)$ denote the number 
of special points in $\D$.  Since $\D$ is a downset,  $s(\D)$ is the number of symmetric chains in  
$\Ch$ that intersect $\D$.\\

\begin{lemma}\label{L:special}
For every binary downset $\D$ we have $w(\D) = s(\D).$
\end{lemma}

\begin{proof}

First observe that $w(\D) \le s(\D)$ because $s(\D)$ is the number of $\C \in \Ch$ which intersect 
$\D$ and $$\{ \C \cap \D \ | \ \C \in \Ch, \ \C \cap \D \ne \emptyset \}$$
is a partition of $\D$ by chains.

To prove the reverse inequality, we first verify the following statement.

{\bf (A).}  For each special point $A \in \D$, say the minimum element of $\C \in \Ch$, let 
$\phi(A)$ be a special point of maximum cardinality in $\D$ such that $A \le_b \phi(A)$.
Then there exists $A' \in \C \cap \D$ such that $|A'| =  |\phi(A)|$ and
$A \le_b A' \le_b \phi(A)$.  See Figure \ref{figure:GK}.

Since $A \le_b A$, $|A| \le |\phi(A)|$. Let $B = \phi(A)$ and $|B| - |A| = r$.   For $r = 0$ then $A = A'$ 
verifies {\bf(A)}; thus we assume $r > 0$.   Then $A <_b B$, which implies that 
$t = \max A \Delta B \in B$.  Representing subsets of $[n]$ as binary sequences, we see that $B$ 
has $r$ more 1's than 0's than $A$ in positions in the interval $[1, t]$.  

\begin{figure}[h]
\begin{center}
\hspace{-1cm}
\scalebox{1.1}{\includegraphics{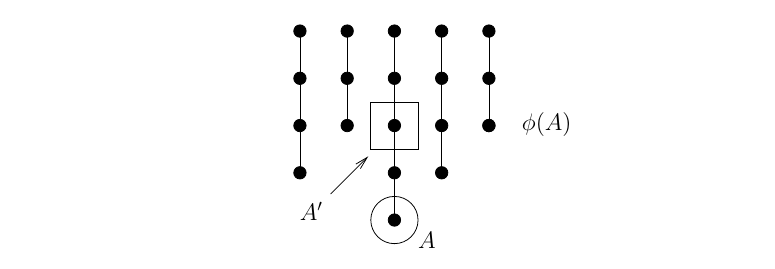}}
\caption{Chains and special points in the proof of $\mathbf{(A)}$.}
\label{figure:GK}
\end{center}
\end{figure}

In the Greene-Kleitman pairing, for a subset of $[n]$,  each 1 is paired with the rightmost unpaired 0 
to its left during the left-to-right pairing process.  Since $A$ and $B$ are minimal members of chains 
in $\Ch$, all 1's in $A$ and $B$ are paired.  Suppose $A$ has $s$ 1's in the interval $[1, t-1]$.  Then, 
according to $B$, 
$$t - 1 = 2(s + r - 1) + 1 + v,$$
where $v$ is the number of unpaired 0's in $B$ in $[1, t - 1]$ and the summand 1 counts the 0 with 
which the 1 in position $t$ of $B$ is paired.  According to $A$,
$$t - 1 = 2s + w,$$
where $w$ is the number of unpaired 0's in $A$ in $[1, t - 1]$.   Since $r \ge 1$ and $v \ge 0$, 
$$2s + w = 2r + 2s + v - 1 \ge 2s + r + v,$$
which implies that $w \ge r$.  We obtain $A' \in \C$ from $A$ by switching $r$ unpaired 0's to 1's, 
from left to right, in the interval $[1, t-1]$ in $A$. Thus $\max A' \Delta B = t$, which means that  
$A' <_b B$ and, therefore, $A' \in \D.$ This completes the proof of {\bf (A).}\\

We now claim that the set of all these $A'$ provide the antichain required.

{\bf (B).} The set $\mathcal{W} = \{ A' : A \in \D \ \text{and is a special point} \}$  is an antichain in $\Bn$.

Since $|\mathcal{W}| = s(\D)$, verifying {\bf (B)} proves the lemma.

Let $A$ and $B$ be special points with $A <_b B$ in $\D$.  Suppose that $A' \subset B'$.  
Then $|A' | < |B'|$ and $|B'| = |C|$ for some special point $C$ in $\D$ such that $B \le_b C$.  
Then $A <_b C$, so $|A'| \ge |C|$ (by the maximality of $|\phi(A)|$), a contradiction.  Now suppose $B' \subset A'$.  By {\bf (A)}, 
there is a special point $D \in \D$ such that $|A'| = |D|$ and $A' \le_b D$.  Since $|D| > |B'|$, we know 
that $D <_b B$, so $A' \le_b D <_b B \le_b B'$, which contradicts $B' \subset A'$.  This proves {\bf (B)} 
and completes the proof of the lemma.
\end{proof}

If we think of building the binary downsets in $\Bn$ sequentially by listing the subsets of $[n]$ in the binary 
order, then the preceding argument shows at which steps the width of the downsets increase.
For all $A \in \Pn$ let $[\emptyset, A) = \{ B \in \Pn : B <_b A \}$ and 
$[\emptyset, A] = [\emptyset, A) \cup \{A\}$.  Of course,
both  $[\emptyset, A)$ and  $[\emptyset, A]$ are downsets in $\Bn$.\\

\begin{proposition}\label{P:special}
For all $A \subseteq [n]$ we have $w([\emptyset, A]) = w([\emptyset, A)) + 1$ if and only if $A$ is a 
special point.
\end{proposition}

\begin{proof}
By Lemma \ref{L:special} the width of a binary downset in $\Bn$,  is the number of chains in $\Ch$ that intersect 
the downset, or the number of special points in the downset. Thus, if we list the subsets of $[n]$ according to the 
binary order $<_b$, the width of the downsets (so enumerated) in $\Bn$ increases exactly when a special point is 
added.
\end{proof}

Conjecture \ref{c:convex}, specialized to downsets $\D$, is that  $w(\D)/|\D|$ is minimized for $\D = \Bn$.   
This is true for binary downsets, as we now show.

\begin{theorem}\label{T:alpha}
For all nonempty binary downsets $\D$ of $\Bn$ we have 
$$w(\D)/|\D| \ge w(\Bn)/|\Bn|.$$
\end{theorem}
\begin{proof}
For any positive integer $d$, let $d= 2^{k_1} + 2^{k_2} + \ldots + 2^{k_s}$, $k_1 > k_2 > \cdots > k_s$, be the binary
representation of $d$.   Our calculations are easier to display if we use the reciprocal, that is, cardinality over width.
So for any binary downset $\D$ of $\Bn$, set $\alpha(\D) = |\D|/w(\D)$.  We show that  $\alpha(\D) \le \alpha(\Bn)$.

Let us proceed by induction on $s$.  If $s = 1$ then $d = 2^{k_1}$ and $\alpha(\D) = 
\alpha(\mathcal{B} (k_1)) \le \alpha(\Bn)$, which follows from the fact that 
$\alpha(\B(k)) = {|\B(k)|}/{w(\B(k))} = 2^k/{\binom{k}{\lfloor k/2 \rfloor}}$ is 
nondecreasing as a function of $k$.

For $s \ge 2$, the binary downset $\C(d)$ of size $d$ has the following partition into intervals of $\Bn$:
$\C(d) = \bigsqcup_{i = 1}^s \Bi$ where\\[-.8cm]

  \begin{equation}\label{e:decomp}
        \Bi = [\{k_1 + 1, k_2 + 1, \ldots, k_{i-1} + 1\}, [k_i] \cup \{k_1 + 1, k_2 + 1, \ldots, k_{i-1} + 1\} ] \cong \mathcal{B}(k_i).
  \end{equation}

Note that $\C(d - 2^{k_s}) =  \bigsqcup_{i = 1}^{s-1} \Bi$ and that \\[-.8cm]

$$\C(d + 2^{k_s}) = \C(d) \bigsqcup [\{k_1 + 1, k_2 + 1, \ldots, k_{s} + 1\}, [k_s] \cup \{k_1 + 1, k_2 + 1, \ldots, k_{s} + 1\} ].$$

Observe that $\C(d + 2^{k_s}) - \C(d)$ is an interval isomorphic to $\mathcal{B}_s$ via the map $X \mapsto X - \{k_s + 1\}$. 

In the Greene-Kleitman SCD $\mathbf{C}$ of $\Bn$, the minimum elements of the members of $\mathbf{C}$, the
special points,  are exactly those sets with no unpaired 1's in the pairing scheme.  This implies that if $X$ is a special point
then every $Y \subseteq X$ is also a special point.  Thus, if $X \in \C(d + 2^{k_s}) - \C(d)$ is special then so is 
$X - \{k_s + 1\} \in \mathcal{B}_s$.  Therefore,\\[-.8cm] 

$$s(\C(d + 2^{k_s})) - s(\C(d)) \le s(\C(d)) - s(\C(d - 2^{k_s})),$$

from which it follows that 
\begin{equation}\label{e:2}
\frac{s(\C(d + 2^{k_s})) + s(\C(d - 2^{k_s}))}{s(\C(d))} \le 2.
\end{equation}

By Lemma \ref{L:special}, $\alpha(\D) = |\D|/s(\D)$.  If\\[-.8cm]

\begin{equation}\label{e:2-in}
\alpha(\C(d)) \ge \alpha(\C(d + 2^{k_s})) \ \ \text{and} \ \ \alpha(\C(d)) \ge \alpha(\C(d - 2^{k_s}))
\end{equation}

and at least one inequality is strict then, using the fact that $|\C(d \pm 2^{k_s})| =  d \pm 2^{k_s}$, we obtain the inequalities
$$\frac{s(\C(d + 2^{k_s}))}{ s(\C(d))} \ge 1 + \frac{2^{k_s}}{d} \quad \text{and} \quad \frac{s(\C(d - 2^{k_s}))}{ s(\C(d))} \ge 1 - \frac{2^{k_s}}{d}$$
with at least one inequality strict.  Add these inequalities, one strict, to contradict (\ref{e:2}).
We now negate (\ref{e:2-in}) with one inequality made
strict and find that at least one of the following holds:\\[-.8cm]

\begin{align}
\alpha(\C(d)) &< \alpha(\C(d - 2^{k_s})),\label{e:one} \\
\alpha(\C(d)) &< \alpha(\C(d + 2^{k_s})), \ \text{or} \label{e:two}\\
\alpha(\C(d)) &= \alpha(\C(d - 2^{k_s})) = \alpha(\C(d + 2^{k_s})).\label{e:three}
\end{align}

To complete the proof, suppose that $d$ is maximum among positive integers less than $2^n$ with $s$ 1's in their
binary representations such that   $\alpha(\C(d)) > \alpha(\Bn)$.  If (\ref{e:one}) or (\ref{e:three}) hold then the induction
hypothesis for integers with $s-1$ 1's in their binary expansion is contradicted.  So assume (\ref{e:two}) holds.
Of course, $d < d +  2^{k_s} \le 2^n$.  If $k_{s-1} = k_s + 1$ then $d +  2^{k_s}$ has at most $s-1$ 1's in its binary
representation, so  $\alpha(\C(d)) > \alpha(\Bn)$ and (\ref{e:two}) contradict the induction hypothesis.  If $k_{s-1} > k_s + 1$
then $d + 2^{k_s}$ has $s$ 1's and we invoke the maximality of $d$:
$$\alpha(\C(d +  2^{k_s})) \le \alpha(\Bn) < \alpha(\C(d)),$$ 

contradicting (\ref{e:two}).  Thus, there is no such $d$, completing the proof by induction.
\end{proof}

As noted, this establishes Conjecture \ref{c:convex} for the collection of binary downsets.   The conjecture for all
downsets would follow from this.

  \begin{conjecture}\label{c:binary-min-width}  
     Among all $d$-element downsets in $\Bn$, the binary $d$-element downset $\C(d)$ has minimum width.     
  \end{conjecture}
  
We now describe the width of $\C(d)$. Goldwasser \cite{G} has independently made Conjecture 8 and proved
Proposition 9.  As usual, if $s < 0$ then $\binom{k}{s} = 0$.\\ 
  
  \begin{proposition}\label{P:Goldwasser}
     Given a positive integer $d$ with binary representation $d = 2^{k_1} + 2^{k_2} + \cdots + 2^{k_r}$ with $k_1 > k_2 > \cdots > k_r \ge 0$ we have
         $$w(\C(d)) = \sum_{i = 1}^r \binom{k_i}{s_i}$$
     where $s_1 = \lceil k_1 / 2 \rceil$ and $s_i = \min (\lceil k_i / 2 \rceil, s_{i-1} - 1)$, $i = 2, 3, \ldots, r$.
  \end{proposition}
  
  \begin{proof}
   We may assume that $n = k_1 + 1$.  Let $K_i = \{k_1 + 1, k_2 + 1, \ldots, k_{i-1} + 1\}$ and use the notation from the proof of 
   Theorem \ref{T:alpha}.  Define $\A$ by
    \begin{equation}\label{e:antichain}    
      \A \ = \ \bigcup_{i = 1}^r \binom{[k_i]}{s_i} \cup K_i  .    
   \end{equation}   
   That is, $\A$ consists of the union of the $s_i^{\mathrm{th}}$ levels  of the Boolean intervals $\Bi$ for all $i = 1, 2, \ldots, r$ for which $s_i \ge 0$ (see \ref{e:decomp})).  Then 
   $\A$ is an antichain of size $\sum_{i = 1}^r \binom{k_i}{s_i}$.   To see that $\A$ realizes $w(\C(d))$, by Lemma \ref{L:special},
   it is enough to prove the following.
   
   {\bf Claim:}  Given a Greene-Kleitman symmetric chain $\C$ of $\Bn$ that has its minimum in $\Bi$, $\C$ intersects level
   $s_i$ of $\Bi$; consequently, $s_i \ge 0$.
   
   Suppose that $\C$ has minimum element $A$.  If $|A| \le s_i + i - 1$ then $|A \cap [k_i]| \le s_i \le \lceil k_i / 2 \rceil$.  Then
   successors of $A$ in $\C$ are obtained by adding elements from $[k_i]$ (that is, switching unmatched 0's to 1's in positions
   1 to $k_i$ in the binary representations of $\C$'s sets) until at least $\lceil k_i / 2 \rceil$ elements from $[k_i]$ are present.  This implies that there is an element of $\C$ 
   that intersects level $s_i$ of $\Bi$.  With $d$ fixed, we proceed by induction on $i$ to prove that: for all $j \ge i$, if
   a Greene-Kleitman chain $\C$ has minimum $A$ in interval $\Bj$ then $|A| \le s_i + i - 1$.
   
   For $i = 1$, every Greene-Kleitman chain $\C$ has minimum $A$ with $|A| \le \lfloor n/2 \rfloor$.  Since $k_1 = n - 1$,
   $\lfloor n/2 \rfloor = \lceil k_1 / 2 \rceil = s_1$.
   
   Let $i  > 1$.   If $s_i = s_{i - 1} - 1$ then $|A| \le s_{i - 1}  + i - 2 = s_i + i - 1$ by the induction hypothesis.  Therefore, we assume
   that $s_i = \lceil k_i / 2 \rceil$ and that a Greene-Kleitman chain $\C$ has minimum $A$ in interval $\Bj$, $j \ge i$.  If $j = i$
   then since $A$ is the minimum of $\C$, there are no unmatched 1's in the binary representation of $A$.  Thus, there are no 
   unmatched 1's in $A \cap [k_i]$.  Then $|A \cap [k_i]| \le \lceil k_i /2\rceil$, so $|A| \le s_i + i - 1$.  Assume $j > i$.  Then
   $A$ is the minimum of a Greene-Kleitman chain so $A \cap [k_i +1] \le \lfloor (k_i + 1)/2 \rfloor$, to avoid unmatched 1's in
   the binary representation of $A$ in positions 1 through  $k_i + 1$.  Since $K_j \subseteq A \subseteq K_j \cup [k_j]$, and 
   $A - [k_i + 1] = \{k_1 + 1, k_2 + 1, \ldots, k_{i-1} + 1\}$, $|A| \le  \lfloor (k_i + 1)/2 \rfloor + i - 1 = \lceil k_i/2 \rceil + i - 1 = s_i + i - 1$.
   
   This completes the induction argument. 
  \end{proof}

\section{Dilworth Partitions of Convex Sets}\label{S:SD}

We are interested in the properties of convex families in $\Bn$, in particular, those that might allow
us to understand the relationship between width and size.  A natural step (underscored by the results in \S\ref{S:binary})
is to consider partitions of convex families by width-many chains. 

Given elements $x < y$ in a partially ordered set $X$, $y$ {\it covers $x$} (or $x$ is a {\it lower cover} of $y$ or $y$ is an {\it upper
cover} of $x$) if $x \le z \le y$ in $X$ implies $z = x$ or $z = y$; denote this by $x \prec y$.  A chain $C$ in $X$ is {\it skipless} if 
$x \prec y$ in $C$ implies $x \prec y$ in $X$.  A partition of $X$ into a family of chains is a {\it Dilworth partition} if there are
$w(X)$ chains in the family.  For brevity, call a Dilworth partition of $X$ into skipless chains an SD-partition of $X$. \\

\begin{theorem}\label{T:SD}
Every convex subset of $\Bn$ has an SD-partition.
\end{theorem}

Although it seems reasonable that convex sets should have SD-partitions, our current proof is a bit involved.   Note that it
is not possible to restrict an arbitrary SD-partition of $\Bn$ to a convex subset $\C$ to obtain an SD-partition of $\C$.
Of course, the restriction will provide a partition of $\C$ into skipless chains.  However the number of chains may exceed
$w(\C)$:  the 4-element set highlighted in $\mathcal{B}_3$ in Figure \ref{F:cube1} has width 2 but intersects 3 chains in
the partition of $\mathcal{B}_3$ given by the dashed-line chains.  

  \begin{figure}
       \scalebox{.4}{\includegraphics{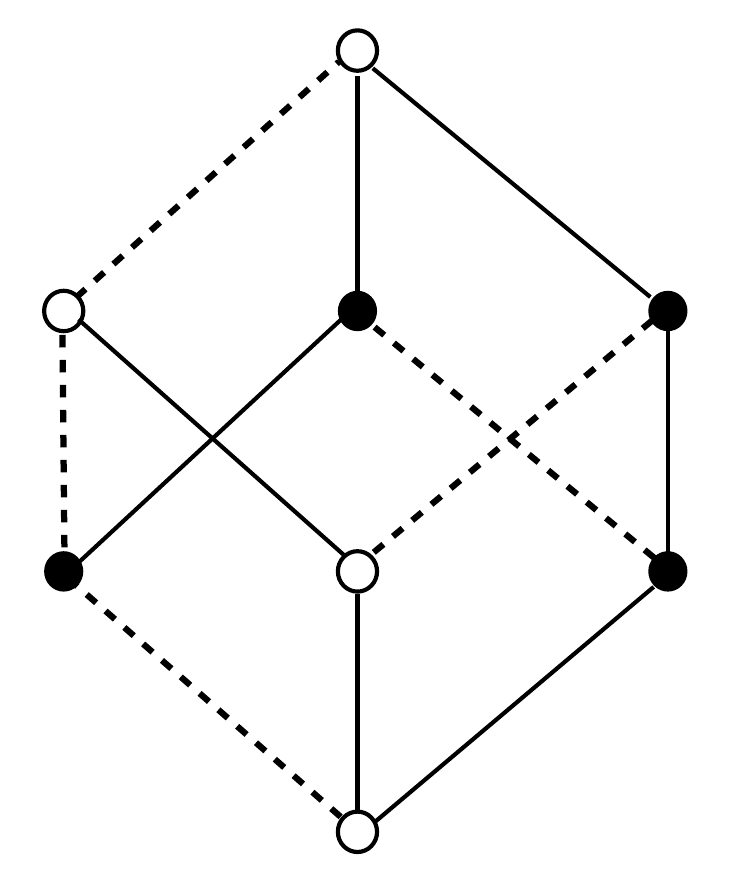}}    
       \caption{A convex subset of $\mathcal{B}(3)$ of width 2 that intersects 3 chains.}\label{F:cube1}    
    \end{figure} 

\begin{proof}
Let $\C$ be a convex subset of $\Bn$ and let $w = w(\C)$.   Proceed by induction on $|\C|$.  For $|\C| = 1$, the result is
obvious.  Induction allows us to assume the following properties.  Recall that a partially ordered set $P$ is {\it connected} if for all
$x, y \in P$ there exist $z_0, z_1, \ldots, z_k \in P$ with $x = z_0$, $y= z_k$, and 
$z_{i - 1}$ is comparable to $z_i$ for $i = 1, 2, \ldots, k$.

{\bf (1)}   $\C$ is connected.  

Otherwise, $\C = \C_1 \cup \C_2$ for disjoint sets $\C_1$ and $\C_2$ with no comparabilities between their elements.
Thus $w(\C) = w(\C_1) + w(\C_2)$ and the union of SD partitions of $\C_1$ and $\C_2$ is an SD partition of $\C$.\\[-.2cm]

{\bf(2)}  If $\mathcal{A}$ is an antichain in $\C$ with $|\mathcal{A}| =  w$ then $\mathcal{A} = \min (\C)$, the
set of minimal elements of $\C$, or $\mathcal{A} = \max (\C)$, the set of maximal elements of $\C$.

This follows from the familiar induction proof of Dilworth's chain decomposition theorem.  If $\mathcal{A}$ is an antichain in 
$\C$ with $|\mathcal{A}| =  w$ and is not contained in either $\min (\C)$ or $\max (\C)$ then
$$\C^* = \{X \in \C \ | \ A \subseteq X \ \text{for some} \ A \in \mathcal{A}\} \ \  \text{and} \ \
    \C_* = \{X \in \C \ | \ X \subseteq A \ \text{for some} \ A \in \mathcal{A}\}$$
are both proper subsets of $\C$ and so, by induction, have SD-partitions.  Since $w(\C_*) = w = w(\C^*)$, the chains in the partition
of $\C_*$ ($\C^*$) each have maximum element (respectively, minimum element) in $\mathcal{A}$.  Thus, we have an SD-partition
of $\C$.\\[-.3cm]

{\bf (3)} We have $|\min(\C)| = w =  |\max(\C)|$.

If  $|\min(\C)| < w$ then $w(\C - \{X\}) = w - 1$ for any $X \in \max(\C)$, so we can add the singleton $\{X\}$ to a $(w-1)$-element
SD-partition of $\C - \{X\}$ to obtain an SD-partition of $\C$.\\[-.2cm]

{\bf (4)}  There do not exist $X \in \min(\C)$ and $Y \in \max(\C)$ with $X \prec Y$.

If such a covering pair exists, we can create an SD-partition of $\C$ from one for $\C - \{X, Y\}$ by including the
covering chain $\{X, Y\}$.\\[-.2cm]

It is convenient to use  $\Le_k$ to denote the set of $k$-element subsets of $[n]$, that is, the $k^{\mathrm{th}}$ level of $\Bn$, $k = 0, 1, \ldots, n$.

{\bf (5)}  There exist $1\le r < t \le n - 1$ such that $\min(\C) = \C \cap \Le_r$,  $\max(\C) = \C \cap \Le_t$ and for all
$r < s < t$, $|\C \cap \Le_s| = w - 1.$ 

To prove {\bf (5)}, we begin with a maximal element $X \in \C$ of maximum cardinality in $\C$, say $|X| = k$.  By induction and noting
$w(\C - \{X\}) = w$ (by  {\bf (1)} and {\bf (3)}), we have an SD-partition $\Ch$ of $\C - \{X\}$ with $w$ chains.  For $S \in \C$, let
$\C(S) \in \Ch$ be the skipless chain containing $S$.  We construct two families of subsets in the bipartite graph $\mathcal{B}$ defined by set containment on the parts $\C \cap \Le_{k-1}$ and $\C \cap \Le_k$.

Let $\Nn_1$ be the set of lower covers of $X_0$ in $\C$, that is, the neighborhood $N(X_0)$ of $X_0$ in $\mathcal{B}$.  For any $X_1 \in \Nn_1$, if $X_1 = \max(\C(X_1))$ then replace $\C(X_1)$ by $\C(X_1) \cup \{X_0\}$
to obtain an SD-partition of $\C$.  Thus, we may assume that for all $X_1 \in \Nn_1$, $X_1 \prec \max(\C(X_1))$.  Let 
$$\A_2 = \{  \max(\C(X_1)) \ | \ X_1 \in \Nn_1 \}.$$
Let $\Nn_3 = N(\A_2) - \Nn_1$.  Given  $X_3 \in \Nn_3$, with$X_3 \prec X_2 = \max(\C(X_1)$ for $X_1 \in \Nn_1$,
if $X_3 = \max(\C(X_3))$ then replace $\C(X_1)$ and $\C(X_3)$ by  $(\C(X_1) - \{X_2\}) \cup \{X_0\}$ and $\C(X_3) \cup \{X_2\}$ in 
$\Ch$ to obtain an SD-partition of $\C$.  Thus, we may assume that $X_3 \prec \max(\C(X_3))$ and let
$$\A_4 = \{  \max(\C(X_3)) \ | \ X_3 \in \Nn_3 \}.$$

Suppose we have constructed the subsets $\A_2, \A_4, \ldots, \A_{2i}$ of  $\C \cap \Le_k$, and $\Nn_1, \Nn_3, \ldots, \Nn_{2i-1}$ of  $C \cap \Le_{k-1}$.  Let $\Nn_{2i+1} = N(\A_{2i}) - (\Nn_1 \cup \Nn_3 \cup \cdots \cup \Nn_{2i-1})$, if this is nonempty, else we stop.  If nonempty, let $X_{2i+1} \in \Nn_{2i+1}$ with a path  $\{X_0, X_1, \ldots, X_{2i}\}$ in $\mathcal{B}$ such that $X_{2j} = \max(\C(X_{2j-1}))$, $X_{2j+1}\in \Nn_{2j+1}$, $X_{2j+1}  \prec X_{2j}$, $j = 1, 2, \dots i$, and $X_1 \in \Nn_1$.  If 
$X_{2i+1} = \max(\C(X_{2i+1}))$ then
$$(\C(X_1) - \{X_2\}) \cup \{X_0\}, (\C(X_3) - \{X_4\}) \cup \{X_2\}, \dots,  \C(X_{2i+1})  \cup \{X_{2i}\}$$
replaces the obvious chains in $\Ch$ to provide an SD-partition of $\C$.   Thus, we may assume that $X_{2i+1} \prec
 \max(\C(X_{2i+1}))$  and  
 $$\A_{2i+2}  = \{  \max(\C(X_{2i+1}) \ | \ X_{2i+1} \in \Nn_{2i+1} \}$$  
 is nonempty.  Consequently, this construction terminates with $\A_{2k}$ for some $k$.
 
 Let $\mathcal{A} = \bigcup_{i = 1}^k \A_{2i}$  and let $\Nn = \bigcup_{i = 1}^{k} \Nn_{2i - 1}$.  Then $\Nn = N(\A)$ since the definition of the $\Nn_i$'s gives $\Nn_{2i-1} \subseteq N(\A)$ and, as $\Nn_{2k+1} =\emptyset$, 
 $N(\A) \subseteq \Nn$.  The map $X \mapsto \max(\C(X))$ is a bijection of $\Nn_{2i-1}$ to $\A_{2i}$ for $i = 1, 2, \ldots, k$.  Therefore, $|\A| = |N(\A)|$.
 
\begin{figure}[h]
\begin{center}
\scalebox{.9}{\includegraphics{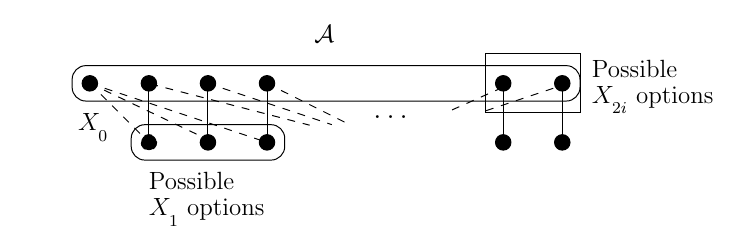}}
\caption{Searching for an alternating path terminating in an $X_{odd}$}
\label{figure:alt}
\end{center}
\end{figure}
 
 Suppose that there exists $Y \in \max(\C) - (\mathcal{A} \cup \{X_0\})$.  By induction, there is an SD-partition of $\C - \{Y\}$ into $w$ chains
 by {\bf (3)}.   Each member $X_{2i}$ of $\mathcal{A} \cup \{X_0\}$ is in such a chain, which requires that a lower cover of $X_{2i}$, 
 a member of $N(\mathcal{A})$, is in the same chain or $\{X_{2i}\}$ is itself a chain in the SD-partition.  Since $|N(\mathcal{A})| =
 |\mathcal{A}|$, there must be a singleton chain.  This contradicts {\bf(3)}.  Therefore, $\max(\C) = \mathcal{A} \cup {X_0}.$
 
 We have shown that  $\max(\C) = \C \cap \Le_t$ (replacing $k$ by $t$), $|\max(\C)| = w$.  The set of lower covers of the maximals
 is $\C \cap \Le_{t-1}$.  This set has size $w - 1$ since $|\C \cap \Le_{t-1}| \le w - 1$ by {\bf(3)} and were $|\C \cap \Le_{t-1}| \le w - 2$
 we would apply induction to $\C - \{X\}$ for any $X \in \max(\C)$ to obtain an SD-partition into $w$ chains.  This partition would have
a chain consisting of a single maximal of  $\C - \{X\}$, but this contradicts $|\min(\C - \{X\})| = w$.  Dually,  $\min(\C) = \C \cap \Le_r$ ($r \le t - 2$ by {\bf (4)}).  The set of upper covers
 of the minimals is $\C \cap \Le_{r + 1}$ and has size $w - 1$.   Consider $\C - \{U, V\}$, where $U \in \min(\C)$ and 
 $V \in \max(\C)$.  This is a convex set with width $w - 1$ and so has a partition into $w - 1$ skipless chains.  Thus, for all
 $r < s < t$, $|\C \cap \Le_s| = w - 1$.  This completes the proof of {\bf(5)}.\\[-.2cm]
 
 With {\bf (5)}, we construct a 3-level convex subset, say $\mathcal{T}$, of $\Bn$, say contained in 
 $\Le_{k-1} \cup \Le_{k} \cup \Le_{k+1}$, that has a partition into $w-1$ 3-element chains, say $C_{i} \subset B_{i} \subset A_{i}$, 
 $i= 1, 2, \ldots, w-1$.  We finish the proof by showing that a 3-level convex subset of width $w-1$ and size $3(w-1)$ does not exist
 in the Boolean lattice.  
 
For each $X \in \mathcal{T}$ let $d^+(X)$ ($d^-(X)$) denote the number of upper (respectively, lower) covers of $X$ in $\mathcal{T}$.
For any $B_i$, $B_i \prec A_j$ if and only if $A_j = B_i \cup \{t\}$.  Then $C_i$ has upper covers $B_i$ and $C_i \cup \{t\}$, so
at least one more than $B_i$.  Thus, $d^+(B_i) < d^+(C_i)$ for each $i$.  Dually, $d^-(B_i) < d^-(A_i)$ for each $i$.   Without
loss of generality, $\sum d^+(C_i) \le \sum d^-(A_i)$.  Then
$$\sum d^+(B_i) < \sum d^+(C_i) \le \sum d^-(A_i) ,$$
which contradicts the fact that $\sum d^+(B_i) =  \sum d^-(A_i)$.
\end{proof}

\section{Width and General Downsets}\label{S:downsets}

Each initial segment of the elements of $\Bn$ listed in the binary order is a downset with respect to the containment order
on $\Bn$.  Proposition \ref{P:special} characterizes the positions in the binary list at which the width of the induced downset 
increases.  We can provide an analogous description for general downsets using SD-partitions, introduced 
for convex subsets in \S\ref{S:SD}.  First, we use alternating paths to give a level-by-level description of SD-partitions.

Let $\D$ be a downset in $\Bn$ that intersects levels $0, 1, \ldots, l$ and let $\Gi$ be the bipartite graph induced by
$\D$ on the parts $\D \cap \Le_{i-1}$ and $\D \cap \Le_i$, $i = 1, 2, \ldots, l$ (where as in the preceding section we denote
level $i$ by $\Le_i$).   Let $\Ch$ be an SD-partition of $\D$.
We claim first that each matching
$$\mathbf{M}_i =  \{\{X, Y\} | X, Y \in \Gi \ \text{and belong to the same chain in} \ \Ch \}$$
is a maximum sized matching in  $\Gi$.   If not, there is an alternating path $\{X_1, Y_1, \dots , X_r, Y_r\}$ in $\Gi$ such
that $\{X_j, Y_j\}$, $j = 1, 2, \ldots, r$, belong to a maximum-sized matching and  
$$\{Y_j , X_{j +1}\} \in \mathbf{M}_i, \ j = 1, 2, \ldots, r-1.$$
Then there are $r+1$ chains in $\Ch$ containing $\{X_1\}, \{Y_1, X_2\}, \ldots , \{Y_{r-1}, X_r\}, \{Y_r\}$ that can be replaced
by $r$ chains containing $\{X_1, Y_1\}, \dots, \{X_r, Y_r\}$.  This gives a partition of $\D$ into fewer than $w(\D)$-many
chains, a contradiction.

On the other hand, let $\mathbf{M}_i$ be a maximum-sized matching in $\Gi$, $i = 1, 2, \dots, l$, and let $\mathcal{M}$
be a graph with vertex set $\D$ and edge set $E(\mathcal{M}) = \cup_{i=1}^l \mathbf{M}_i$.  As a graph, the connected components of 
$\mathcal{M}$ are just paths.  As an ordered set $\mathcal{M}$ provides a partition of $\D$ into
skipless chains.  By the preceding paragraph, the number $e(\mathcal{M})$ of edges in  $\mathcal{M}$ is the same
as the number in an SD-partition of $\D$.  Thus, $|E(\mathcal{M})| = |\D| - w(\D)$ since the number of components of an
SD-partition, regarded as a graph, is the number of chains in a Dilworth partition of $\D$, the width of $\D$.  It
follows that the number of connected components of $\mathcal{M}$ is $w(\D)$ and that  $\mathcal{M}$ is an
SD-partition of $\D$.

Alternating paths allow us to prove something along the lines of Proposition \ref{P:special} for arbitrary downsets.
Let $\D$ be a downset in $\Bn$ and let $\Ch$ be an SD-partition of $\D$.  Suppose that $Y \in \Bn - \D$ with $|Y| = k$
and that $\D' = \D \cup \{Y\}$ is also a downset.   Let $\Gk$ be the bipartite graph induced by $\D'$ on parts
$\D' \cap \Le_{k-1}$ and $\D' \cap \Le_k$ and let $\mathbf{M}$ be the matching in $\Gk$ consisting of the edges of
$\Ch$ in $\Gk$.   

A path $\{Y, X_1, Y_1, X_2, Y_2, \ldots, X_r, Y_r, X_{r+1}\}$ in $\Gk$ such that each 
$\{X_i, Y_i\} \in \mathbf{M}$ and $X_{r+1}$ is the maximum element of its chain in $\Ch$ is called {\it augmenting}.\\

\begin{proposition}\label{P:general}
With the preceding notation, $w(\D) = w(\D')$ if and only if there is an augmenting path in $\Gk$.  
\end{proposition}

\begin{proof}
Given an augmenting path $\{Y, X_1, Y_1, X_2, Y_2, \ldots, X_r, Y_r, X_{r+1}\}$ in $\Gk$, the $r + 1$ skipless chains
in an SD-partition of $\D$ each containing one of
$$\{X_1, Y_1\}, \{X_2, Y_2\}, \ldots, \{X_r, Y_r\} \  \text{and} \  \{X_{r+1}\}$$
can be replaced by $r + 1$ skipless chains each containing one of
$$\{X_1, Y\}, \{X_2, Y_1\}, \ldots, \{X_r, Y_{r - 1}\} \  \text{and} \  \{X_{r+1}, Y_r\}$$
to create an SD-partition of the same size for $\D'$.

To prove the converse, we assume that $w(\D) = w(\D')$ and that there is no augmenting path in $\Gk$.  We have
an SD-partition $\Ch$ of $\D$; let $\Ch'$ be an SD-partition of $\D'$.   We consider alternating paths in $\Gk$ using
only edges from the chains in $\Ch$ or $\Ch'$ that begin with edges $\{Y, X_1\}$ from a chain in $\Ch'$.  There must
be such an edge since otherwise $Y$ belongs to a 1-element chain in $\Ch'$ which would contradict $w(\D) = w(\D')$.
If $X_1$ is maximum in its $\Ch$ chain, stop; otherwise add an edge $\{X_1, Y_1\}$ from a chain in $\Ch$.  If $Y_1$
is minimum in its $\Ch'$ chain, stop; otherwise, add an edge $\{Y_1, X_2\}$ from a chain in $\Ch'$.   Continue in the
same manner.  Since there is no augmenting path, this process must terminate with an edge $\{X_r, Y_r\}$ from $\Ch$.  

Now replace the $r + 1$ skipless chains from $\Ch'$ that each contain one of
$$\{Y, X_1\}, \{Y_1, X_2\}, \ldots, \{Y_{r - 1}, X_r\} \  \text{and} \  \{Y_{r+1}\}$$
with $r+1$ skipless chains that each contain one of
$$\{Y\}, \{X_1, Y_1\}, \{X_2, Y_2\}, \ldots, \{X_r, Y_r\}.$$
Because $Y$ is maximal in $\D'$, the resulting SD-partition of $\D'$ has a singleton chain $\{Y\}$.  This again
contradicts $w(\D) = w(\D')$.
\end{proof}

\section{Maximizing the generated downset}\label{S:shadows}
Call a family of $r$-sets {\it top-heavy} or simply {\it heavy} if there is no larger
antichain in the downset it generates.  We would like to answer this question: among heavy families of $r$-sets of 
given size, which one generates the largest downset?  We formulate a conjecture, verify it for the first nontrivial
case and prove a (rather weak) bound.  For $\mathcal{T} \subseteq \Bn$, let $\dw \mathcal{T}$ denote the
downset generated by $\mathcal{T}$.  We use the standard shadow notation in the case of downsets $\D$
of $\Bn$: given a family $\mathcal{T}$ of $k$-sets in $\D$, let $\Delta(\mathcal{T})$ be the set of all $Y \in \ \dw \!\mathcal{T} \cap \D$
with $|Y| = k - 1$.

\begin{conjecture} \label{topheavyconj}
Let $\mathcal{T}$ be a heavy family of $t$ $r$-sets.  Then
$$|\dw \mathcal{T}| \ \le \ \left[\frac{2^{2r-2}-1}{\binom{2r-1}{r}}+1\right] t \ +  \ 1 \ .$$
\end{conjecture}

Let $f_r(t)$ be the maximum size of a heavy downset generated by $t$ $r$-sets.

Here are some straightforward observations about this function.  First, $f_r(t)$ is not defined for small values of $t$.  
For instance,
if $t < \binom{r}{\lfloor r/2\rfloor}$, then one maximal element of any downset of height $r$ contains $\binom{r}{\lfloor r/2\rfloor}$ subsets
of size $\lfloor r/2\rfloor$, an antichain in the downset; thus, no heavy downset of height $r$ and width  $t < \binom{r}{\lfloor r/2\rfloor}$
can exist.  Second, $f_r(t) \leq rt+1$ since every level of a top-heavy downset has 
size at most $r$ and level 0 consists of $\emptyset$.  Third, 
$$f_1(t)=t+1\ \text{and} \  f_2(t)=2t+1$$
follow from simple constructions.

Provided that $t=k \cdot \binom{2r-1}{r}$ if the conjecture is correct then it would be tight by the following construction. 
Suppose that $X_i$, $i = 1, 2, \ldots, k$, are pairwise disjoint subsets of $[n]$, with each $|X_i | = 2r-1$.  Let the downset 
$$\HH \ = \ \{ A \ | \ |A| \le r \ \text{and} \ A \subseteq X_i \ \text{for some} \ i = 1, 2, \ldots, k \}.$$
Thus,  $\HH$ is the union of $k$ copies of the first $r$ levels of $\mathcal{B}(2r - 1)$.  Each copy only has the empty set as a common element in the union. Furthermore, the width of this downset is $k \cdot \binom{2r-1}{r}$ and the number of elements is $k \cdot (2^{2r-1}/2-1+\binom{2r-1}{r})+1$.

We note that:
\[\left[\frac{2^{2r-2}-1}{\binom{2r-1}{r}}+1\right]t + 1 \ \approx \ t\sqrt{\frac{\pi(2r-1)}{8}} \ = \ \Theta (t\sqrt{r}) .\]

We can improve the trivial upper bound  $f_r(t) \leq rt+1$ by about a third by showing that the total number of elements at height $2r/3$
in a heavy downset of height $r$ and width $t$ is less than $ 4 t \sqrt{r}$.

\begin{proposition}\label{P:f_h}
$f_r(t) \leq  t(r/3 + 4 \sqrt{r})$
\end{proposition}
\begin{proof}
Let $\HH$ be a heavy downset generated by $r$-sets.   Let $X \in \HH$ be a $k$-set, $k < r$.  We first find a lower bound on the number of upper covers of $X$ in $\HH$.

{\bf Claim 1.}  There are at least $2(r - k) - 1$ upper covers of $X$ in $\HH$.\\[-.3cm]

Let $\HH_i$ denote the family of sets in $\HH$ at level $i$ in $\Bn$ (that is, the $i$-subsets of $[n]$ in $\HH$) and let $\up X$ denote the upset  generated by $X$ in $\Bn$.  Then
$$| \HH_i  \ \cap \up X| \ \le \ | \HH_r  \ \cap \up X|$$
as otherwise $ (\HH_r  -\! \up X) \cup (  \HH_i  \cap \! \up X)$ is an antichain in $\HH$ of size greater than $t$, a contradiction.

Because the maximal elements of $\HH$ are $r$-sets, $\HH_{r-1} \cap \up X = \Delta(\HH_r \cap \up X) \cap \up X$.  Since $\HH$ is heavy,  
 \begin{equation}\label{e:top}
 |\HH_r \cap \up X| \ge  |\Delta(\HH_r \cap \up X) \cap \up X| .
 \end{equation}

Observe that $\Delta(\HH_r \cap \up X) \cap \up X$ is the shadow of $\HH_r \cap \up X$ in $\up X$ where  $\up X$ is isomorphic to the Boolean lattice $\mathcal{B}(n - k)$.   The Kruskal-Katona theorem (\cite{Ka}, \cite{Kr}) shows that if a family $\F$ of $(r-k)$-element sets in $\mathcal{B}(n - k)$ has
$|\F| < \binom{2(r - k) - 1}{r - k}$ then $|\Delta(\F)| > |\F|$.  In view of (\ref{e:top}), we have that
$$|\HH_r \cap \up X| \ge \binom{2(r - k) - 1}{r - k} .$$
Since each set in $\HH_r \cap \up X$ is the union of $(k+1)$-element sets containing $X$, all of which must be members of 
$\HH$, $X$ has at least $2(r - k) -1$ upper covers in $\HH$.  This verifies Claim~1.

Every $k+1$-set in $\HH$ covers exactly $k+1$ members of $\HH_k$.  Each set in $\HH_k$ is covered by at least  $2(r - k) -1$ in $\HH_{k+1}$ for $k = 0, 1, \ldots, r-1$.  Counting the edges in the bipartite containment graph induced by levels $k$ and $k+1$ of $\HH$ verifies\\
\begin{equation}\label{e: consecutive}
| \HH_k | \ \le \ \frac{k+1}{2(r - k) - 1} \cdot | \HH_{k+1} |, \  \ k = 0, 1, \ldots, r - 1 .\\
\end{equation}

{\bf Claim 2.}
For $0 \le i \le  \frac{2}{3}r-2\sqrt{r}$, $| \HH_i|  \le \frac{1}{2} | \HH_{i + 2\sqrt{r}}|$.\\[-.3cm]

To verify this, first observe that repeated application of the inequality in (\ref{e: consecutive}) shows that $| \HH_i|  \le c_i  | \HH_{i + 2\sqrt{r}}|$ where
\begin{equation}\label{e:constants}
c_i = \frac{i+1}{2(r - i) - 1} \cdot \frac{i+2}{2(r - i) - 3}  \cdot \frac{i+3}{2(r - i) - 5} \cdots \frac{i+2\sqrt{r}}{2(r - i) - 4\sqrt{r} + 1} \ .
\end{equation}
By comparing terms in these constants, we see that $c_i \le c_{\frac{2}{3}r-2\sqrt{r}}$ for $i = 0, 1, \ldots,  \frac{2}{3}r-2\sqrt{r}$.
For each $j = 1, 2, \ldots, \sqrt{r}$, and $i =  \frac{2}{3}r-2\sqrt{r}$, the $j^{\mathrm{th}}$ factor in (\ref{e:constants}) is bounded above
as follows:
$$\frac{ 2r/3 - 2\sqrt{r} + j}{2r/3 + 4\sqrt{r} - (2j-1)} \ \le \ \frac{2r/3 - \sqrt{r}}{2r/3 + 2\sqrt{r}}.$$
For each $j =  \sqrt{r} +1, \sqrt{r} + 2, \ldots, 2\sqrt{r}$, the $j^{\mathrm{th}}$ factor in  (\ref{e:constants}) is bounded above by 1.  Therefore, for
each $i = 0, 1, \ldots,  \frac{2}{3}r-2\sqrt{r}$, 
$$c_i \ \le \ c_{\frac{2}{3}r-2\sqrt{r}}\ \le \ \left(\frac{2r/3 - \sqrt{r}}{2r/3 + 2\sqrt{r}}\right)^{\sqrt{r}} .$$
Claim 2 now follows from the fact that
$$\left(\frac{2r/3 - \sqrt{r}}{2r/3 + 2\sqrt{r}}\right)^{\sqrt{r}} = \left(1 - \frac{3\sqrt{r}}{2r/3 + 2\sqrt{r}}\right)^{\sqrt{r}} \le 
\exp\left(-\frac{3r}{2r/3 + 2\sqrt{r}}\right) < \frac{1}{2} .$$

Partition the bottom $2r/3$ levels of $\HH$ into sets of $2\sqrt{r}$ consecutive levels.  Use the trivial bound of $t$ for each of the $2\sqrt{r}$
levels in the first part, that is, the top part of the bottom  $2r/3$ levels of $\HH$. Then Claim 2 shows that the total size of the $j^{\mathrm{th}}$ set of $2\sqrt{r}$ consecutive levels is bounded above by $(1/2)^{j-1} \cdot 2\sqrt{r} \cdot t$.  Thus, the bottom  $2r/3$ levels of $\HH$ have total size bounded above by $4\sqrt{r} \cdot t$.  Bounding the size of each of the top $r/3$ levels by $t$ completes the proof. \end{proof}

\begin{proposition}\label{P:r=3}
Conjecture \ref{topheavyconj} is true for $r = 3$, namely $f_3(t)  \le  2.5t + 1.$
\end{proposition}
\begin{proof}
Let $\HH$ be a top heavy downset of height 3.  We claim that it is enough to prove that the average number of upper covers
of a singleton in $\HH$ is at least 4.   Suppose this holds.  Since each member of $\HH_2$ covers two members of $\HH_1$, we would have
$|\HH_2| \ge 2 |\HH_1|$.  Therefore, $|\HH| = \sum_{i = 0}^3 |\HH_i | \le t + t + (1/2)t + 1 = 2.5 t + 1$.

Let $X\in \HH_1$, say $X$ = \{1\}.  Since $\HH$ is heavy, its maximals each have size 3, so $X \subset Y = \{1, 2, 3\}$.  Then $X$ has upper covers $\{1, 2\}$ and $\{1, 3\}$.  Since $\HH$ has width equal to the number of its maximals, $X \subset  Z$ for $Z \ne Y$ and $|Z| = 3$.  At least one 2-element subset of $Z$ gives a third upper cover of $X$.

Suppose that $X$ has exactly 3 upper covers.  Then $\up X$ consists of $X$, 3 upper covers and 3 3-element sets -- otherwise $X$
has more than 3 upper covers.   Without loss of generality,  
$$\up X = \{ \{1\}, \{1, 2\}, \{1, 3\}, \{1,4\}, \{1, 2, 3\}, \{1, 2, 4\}, \{1, 3, 4\}\}.$$

Consider $\HH - \up X$.  We claim that each singleton $Y \in \HH - \up X$ has at least 3 upper covers in $\HH - \up X$.  If not, without loss
of generality, we may
take $Y = \{4\}$ and its upset in $\HH$ is
$$\up Y = \{ \{4\}, \{2,4\}, \{3, 4\}, \{1,4\}, \{2, 3, 4\}, \{1, 2, 4\}, \{1, 3, 4\} .$$
In particular, there is exactly one maximal above $Y$ that is not above $X$.  Then the 5 2-element sets containing $X$ or $Y$ together 
with the maximals of $\HH$ not above $X$ or $Y$ would be a larger antichain than $\max(\HH)$, a contradiction.

Let $X_1 = X$.  If some $Y$ has exactly 3 upper covers in $\HH - \up X_1$, let $X_2 = Y$, otherwise we stop with all singletons in
$\HH - \up X_1$ with at least 4 upper covers in $\HH - \up X_1$. 

Suppose that we have a sequence of singletons $X_1, X_2, \ldots, X_k$ such that each $X_j$ has exactly 3 upper covers in 
$\HH - \bigcup_{i = 1}^{j-1} \! \up X_i$.  As above, $\up X_j - \bigcup_{i = 1}^{j-1} \! \up X_i$ consists of $X_j$, 3 2-element sets and
3 3-element sets.   We again argue that each $Y \in \HH -  \bigcup_{i = 1}^{k} \! \up X_i$ has at least 3 upper covers in $\HH -  \bigcup_{i = 1}^{k} \! \up X_i$.  If some $Y$ does not, because $Y$ has at least 3 upper covers in $\HH$, $Y = \{y\}$ is contained in a 2-element set
in some $\up X_i$, say $j$ is the maximum such, $X_j = \{x\}$, and $\{x, y\} \in \up X_j$.  Then $\{u, x, y\}, \{v, x, y\}  \in \up X_j$.  The
only possible 3-element sets in $\HH$ that contain $Y$ are in $\bigcup_{i = 1}^{k} \! \up X_i$ or equal $\{y, u, v\}$.  It follows that
$$\left(\left(\HH_2 \cap \bigcup_{i = 1}^{k} \! \up X_i\right) \cup \{u, y\} \cup  \{v, y\}\right) \cup \left(\HH_3 - \up \{X_1, X_2, \ldots,X_k, Y\}\right)$$
is an antichain larger than the set of maximals of $\HH$, a contradiction.

If some $Y$ has exactly 3 upper covers in $\HH - \bigcup_{i = 1}^{k} \up \! X_i $, let $X_{k+1} = Y$, otherwise we stop with all singletons in $\HH - \bigcup_{i = 1}^{k}\! \up X_i $ with at least 4 upper covers in $\HH - \bigcup_{i = 1}^{k}\! \up X_i$.

This procedure stops after $s$ steps.  Consider the edge set in the bipartite containment graph induced by $\HH_1 \cup \HH_2$.  There are $3s$ members of $\HH_2$ in $\bigcup_{i = 1}^{s}\! \up X_i$ that account for $6s$ edges in this graph.  Each of the $m$ singletons in  $\HH - \bigcup_{i = 1}^{s}\! \up X_i$
is incident with at least 4 edges, none of which are incident with the $3s$ 2-element sets in $\bigcup_{i = 1}^{s}\! \up X_i$.  This gives
a total of at least $6s + 4m$ edges incident with exactly $s + m$ members of $\HH_1$.  Thus, the average number of covers of singletons
in $\HH$ is at least  $(6s+4m)/(s+m)\geq 4$.
\end{proof}

We think that it is unlikely that the approach of Proposition \ref{P:r=3} will work for $r > 3$.


\end{document}